\documentclass[11pt,oneside,leqno]{amsart}
\usepackage{amsxtra}
\usepackage{amsopn}
\usepackage{amsmath,amsthm,amssymb}
\usepackage{amscd}
\usepackage{amsfonts}
\usepackage{latexsym}
\usepackage{verbatim}
\usepackage{pb-diagram}

\newcommand{\fo}{\mathcal{F}}

\theoremstyle{plain}
\newtheorem{theorem}{Theorem}[section]
\newtheorem*{theorem*}{Theorem}
\newtheorem{definition}[theorem]{Definition}

\newtheorem{prop}[theorem]{Proposition}
\newtheorem{cor}[theorem]{Corollary}
\newtheorem{rem}[theorem]{Remark}
\newtheorem{ex}[theorem]{Example}
\begin{document}
\title{Some results on  cosymplectic manifolds}
\author{Anna Fino and Luigi Vezzoni}
\date{\today}

\address{Dipartimento di Matematica \\ Universit\`a di Torino\\
Via Carlo Alberto 10 \\
10123 Torino\\ Italy} \email{annamaria.fino@unito.it, luigi.vezzoni@unito.it}
\thanks{This work was supported by the Project M.I.U.R. ``Riemann Metrics and  Differenziable Manifolds''
and by G.N.S.A.G.A.
of I.N.d.A.M.}
\begin{abstract}
We  obtain a generalization of  the Kodaira-Morrow stability theorem for  cosymplectic structures. We investigate cosymplectic geometry on Lie groups and on  their compact quotients by  uniform discrete subgroups. In this way we   show that a compact  solvmanifold admits a cosymplectic structure if and only if it is a finite quotient of a torus.
\end{abstract}
\maketitle
\newcommand\C{{\mathbb C}}
\newcommand\f{\mathcal{F}}
\newcommand\g{{\frak{g}}}
\renewcommand\k{{\kappa}}
\renewcommand\l{{\lambda}}
\newcommand\m{{\mu}}
\renewcommand\O{{\Omega}}
\renewcommand\t{{\theta}}
\newcommand\ebar{{\bar{\varepsilon}}}
\newcommand\R{{\mathbb R}}
\newcommand\Z{{\mathbb Z}}
\newcommand\T{{\mathbb T}}
\newcommand{\de}[2]{\frac{\partial #1}{\partial #2}}
\newcommand\w{\wedge}
\newcommand{\ov}[1]{\overline{ #1}}
\newcommand{\Tk}{\mathcal{T}_{\omega}}
\newcommand{\ovp}{\overline{\partial}}
\section{Introduction}

 An odd-dimensional counterpart of a K\"ahler manifold is given by a  cosymplectic manifold, which is locally a product of a K\"ahler manifold with a circle or a line. Indeed,
a {\emph {cosymplectic} structure on   a $(2n + 1)$-dimensional manifold $M$ is  a  normal almost contact  metric structure $(J,\xi,\alpha,g)$ on $M$   such that  the $1$-form $\alpha$ and   the fundamental $2$-form $\omega$ are closed  (see \cite{Blair, BlairGo,CLM}).

In the context of CR geometry,   cosymplectic structures can  be
also  viewed as a special class of \emph{Levi-flat CR-structures}
since the foliation  $\ker\alpha$,  endowed with the complex
structure $J$, defines a  complex subbundle $\mathcal{F}^{\C}$ of
the complexified tangent bundle  $TM\otimes\C$ satisfying $
\mathcal{F}^{\C}\cap\overline{\mathcal{F}^{\C}}=0$ and such that the
spaces  of sections $\Gamma(\mathcal{F}^{\C})$ and
$\Gamma({\mathcal{F}}^{\C} \oplus\overline{\mathcal{F}^{\C}})$ are
both closed under the Lie brackets (for more details on CR
structures   see for instance \cite{DG}).

Trivial  examples of  cosymplectic manifolds  are given  by  a
product of a  $2n$-dimensional K\"ahler manifold with  a
$1$-dimensional manifold and by  products of the $(2m +
1)$-dimensional real torus ${\mathbb T}^{2 m + 1}$ with the
$r$-dimensional complex projective space $\C {\mathbb P}^r$, where
$m, r \geq 0$ and $m + r =n$.   Not many examples of  non-trivial
cosymplectic manifolds are known.  In \cite{CLM}  an example of
$3$-dimensional compact cosymplectic manifold not topologically
equivalent to a global product of a compact K\"ahler manifold with
the circle  was given. Moreover,  in \cite{marrero} Marrero and
Padron found some examples of $(2n + 1)$-dimensional  compact
cosymplectic manifolds which are not topologically equivalent to the
standard example  ${\mathbb T}^{2 m + 1} \times \C {\mathbb P}^r$.
These manifolds are  constructed as suspensions with fibre a compact
K\"ahler manifold of representations defined by Hermitian isometries and they are compact
solvmanifolds, i.e. compact quotients of a simply-connected solvable
Lie group by a  uniform discrete  subgroup.  We will show that in
fact such solvmanifolds are finite quotients of tori.

The presence of so few  cosymplectic compact examples is in part due to many  topological restrictions. Indeed,
in \cite{BlairGo,CLM, Fujitani, CDM, LM} some  topological properties of compact  normal contact metric  manifolds and compact cosymplectic manifolds are shown. For instance, if $(M,J, \xi,\alpha,g)$ is a compact
$(2n+1)$-dimensional cosymplectic manifold, then the Betti numbers $b_i(M)$ satisfy the  conditions:
\begin{enumerate}
\item $b_i(M)$ are non-zero for all $0\leq i\leq 2n+1$;
\item $b_0(M)\leq b_1(M)\leq\dots \leq b_n(M)=b_{n+1}(M)$;
\item $b_{n+1}(M)\geq b_{n+2}(M)\geq\dots \geq b_{2n+1}(M)$.
\vspace{0.2cm}
\end{enumerate}

\noindent
Moreover, like for  the K\" ahler case,
 a compact cosymplectic manifold $M$ is formal in the sense of Sullivan \cite{sullivan}.

In order to find new cosymplectic structures  we will study small deformations of  cosymplectic structures and  determine which types of Lie groups are cosymplectic.

More precisely, in the first part of the paper  we will examine how
a  cosymplectic structure $(J,\xi,\alpha,g)$ on a  compact $(2n+
1)$-dimensional manifold changes under small deformations of the
complex structure  $J$ on $\ker \alpha$. To this aim we will
interpret a cosymplectic structure in terms of smooth foliations on
the manifold. Indeed, if $(M,J,\xi,\alpha,g)$ is a cosymplectic
manifold, then the pair $(\ker\alpha,\omega)$ defines a
K\"ahler-Riemann foliation on $M$. Therefore it is possible to
apply the results
   obtained in \cite{ALK,D} about K\"ahler-Riemann foliations, that we  will review in Section 2 and 3.
   By  using similar methods to the ones  introduced  by  Kodaira and Morrow for the stability theorem of K\"ahler manifolds (see Theorem
4.6 of \cite{KM}),  in Section 5 we will  obtain the following

\begin{theorem}\label{main}
Let $(M,J,\xi,\alpha,g)$ be a compact cosymplectic manifold and let $J_t$ be a  family of endomorphisms of $\,TM$  depending differentiably on $t$ and satisfying
$$
J_t^2=-{\rm Id}+\alpha\otimes \xi\,,\quad J_0=J\,,  \quad N_{J_t} =0\,,
$$
where $N_{J_t}$ denotes the  Nijenhuis tensor of $J_t$.
Then there exists, for $\vert t\vert$ small, a Riemannian metric $g_t$ on $M$ such that $g_0=g$ and
 $(J_t,\xi,\alpha,g_t)$ defines a cosymplectic structure on $M$.
\end{theorem}

In the last two sections  we will study  left-invariant cosymplectic structures on Lie groups and  on their compact quotients  by  uniform discrete subgroups.  Summing up all the results  we will get  the following

\begin{theorem}
There exists  a one-to-one correspondence between cosymplectic and K\"ahler Lie algebras.\newline
A cosymplectic unimodular
Lie group is  necessarily flat and solvable. Moreover,  if a solvmanifold admits a cosymplectic structure,  then it  is a finite quotient of a torus.
\end{theorem}

Although the first part of this theorem comes from a result of Dacko (see \cite{D1}), for sake of completeness
we prove it in section 6 (see the proof of Theorem 6.1).
More precisely, from \cite{D1}, we have that if  a  $2n$-dimensional K\"ahler Lie algebra has  a particular type of derivation, then one can construct   a cosymplectic  Lie algebra of dimension $2n + 1$. Since by
 \cite{Dorfmeister} any solvable  K\"ahler Lie algebra can be obtained by using   modifications and  normal $J$-algebras, it turns out that  the previous construction can be applied  to solvable K\"ahler Lie algebras.

The fact  that  a cosymplectic unimodular
Lie group is flat and solvable   follows from the results in  \cite{LM} and  \cite{Hano}   about symplectic  and K\" ahler Lie groups. Therefore it is natural to investigate which types of solvmanifolds admit a cosymplectic structure. Like in the K\"ahler case  (see  \cite{Ha})  the cosymplectic condition imposes strong restrictions.

\section{Riemannian foliations}
In this section we will recall  some   properties about Riemannian foliations that we will use in the next sections (for more details see for instance \cite{ALK,Molino1,Molino2,Reinhart}).

\begin{definition}\emph{ A smooth   foliation $\fo$ on  a  compact Riemannian manifold $(M,g)$ is called a \emph{Riemannian foliation} if
the Riemannian metric $g$ is \emph{bundle-like}, i.e.  if the normal plane field to $\fo$ is totally geodesic. }
\end{definition}
In other words a Riemannian foliation is a foliation which is locally defined by Riemannian submersions.
Such foliations have been introduced by  Reinhart \cite{Reinhart}  and studied subsequently  by  Molino \cite{Molino1,Molino2}.

Consider now
a $\Z$-graded  Riemannian vector  bundle $E$ on $M$ and denote by $\Gamma(E)$  the space of the its smooth sections. Let $d\colon \Gamma(E)\to \Gamma(E)$
be a first-order leafwise elliptic  differential operator
and denote by $\delta$  the formal adjoint of $d$ with respect to the Riemannian metric $g$. Then the Laplacian
$$
\Delta:=(d+\delta)^2
$$
is a self-adjoint  differential operator on $\Gamma (E)$. We recall that a
\emph {leafwise differential operator} on $\Gamma(E)$ is a differential operator whose local expressions only contain
 derivatives
along leaf directions of $\f$ and it can be restricted to the leaves. If in addition  the restriction to the leaves is elliptic, then the leafwise differential operator is called  \emph{leafwise elliptic}. Moreover, a \emph{transversely elliptic} differential operator is a differential operator  whose leading symbol is an isomorphism at non-trivial
covectors normal to the leaves, where by normal to the leaves we mean that  the covector  vanishes on vectors tangent to the leaves.

In the case of transversely elliptic differential operators Alvarez L\'opez and Kordyukov  in  \cite{ALK}  proved that if  $A\colon \Gamma(E)\to\Gamma(E)$ is a transversely elliptic first order differential operator and  there exist morphisms
$G,H,K$ and $L$ such that
$$
Ad\pm dA=Gd+dH\,,\quad A\delta\pm \delta A=K\delta+\delta L\,,
$$
then the orthogonal projection $\pi\colon\Gamma(E)\to\ker\Delta$ is a continuous operator on $\Gamma(E)$ and
\begin{equation}
\label{ALK1}
\Gamma(E)=\ker\Delta\oplus\overline{{\rm im}\, d}\oplus\overline{{\rm im}\, \delta}\,,
\end{equation}
where the bar denotes the closure in the Fr\'echet space $\Gamma(E)$.

\noindent In this way one gets a natural splitting of $\Gamma(E)$ in terms of the Laplacian $\Delta$ and by \cite{ALK} it is possible to apply  the previous result to the case  of Riemannian foliations, constructing a leafwise differential complex of $\f$ as follows.

Let $\Omega$ be the de Rham cohomology algebra of $M$.
Denote by  $T\f^*$  the dual bundle of  the tangent bundle $T\f$ with respect to the duality
induced by the metric $g$ and by $T\f^{\perp}$ the bundle orthogonal to $T\f$. One has the  following bigrading of the algebra $\Omega$
\begin{equation}
\label{forme}
\Omega^{u,v} =\Gamma\Big(\bigwedge^u T\f^{\bot*}\otimes\,\,\bigwedge^{v}T\f^*\Big)\,,\quad u,v\in\Z\,.
\end{equation}
Since
$$
d(\Omega^{u,v}) \subseteq \Omega^{u+1,v}\oplus\Omega^{u,v+1}\oplus\Omega^{u+2,v-1}\,,
$$
then the  de Rham differential $d$  and its codifferential $\delta$ decompose respectively as
$$
d=d_{0,1}+d_{1,0}+d_{2,-1}\,,\quad
\delta=\delta_{0,-1}+\delta_{-1,0}+\delta_{-2,1}\,,
$$
where the double subindices denote the corresponding bidegrees of
the bihomogeneous  components. Each $\delta_{i,j}$ is the formal
adjoint of $d_{-i, -j}$ and if we denote by $D_0 = d_{0,1} +
\delta_{0,-1}$, then $D_0$ and  the Laplacian $\Delta_0 = D_0^2$ are
leafwise elliptic and symmetric  differential operators.

By applying  \eqref{ALK1}  to  $(\Omega, d_{0,1})$ Alvarez L\'opez and Kordyukov   showed the following result
\begin{theorem}[\cite{ALK}, Theorem B]\label{ALK2} Assume that $\fo$ is a Riemannian foliation on a compact manifold  $M$ with a bundle-like metric; then,
with the notation stated above,
the following leafwise Hodge decomposition
\begin{equation}\label{HodgeF}
\Omega=\ker  \Delta_0  \oplus \overline{{\rm im}\,  \Delta_0} = (\ker d_{0,1} \cap \ker \delta_{0,-1}) \oplus\overline{{\rm im}\, d_{0,1}}\oplus \overline{{\rm im}\, \delta_{0,-1}}
\end{equation}
holds.
\end{theorem}
\medskip

Note that the   operator
$$
d_{0,1}: \Omega^{0,v}\to\Omega^{0,v+1}
$$
can be canonically identified with the de Rham derivative $d_\fo$ on the leaves, or equivalently it can be  defined by
$$
(d_\fo \, \alpha)_{|l}=d(\alpha_{|l}),
$$
for any $\alpha\in\Omega^{0,v}$ and any leaf $\,l$ of $\fo$. Moreover, $(\Omega, d_{0,1})$ can be considered, up to a sign, as the leafwise de Rham complex of $\fo$ with coefficients in the vector bundle $\Lambda (TM / T \fo)^*$.

 Furthermore we denote by
\begin{equation} \label{laplacianf}
\delta_\fo=-*_\fo d_\fo  *_\fo\, , \quad \Delta_{\fo} = (d_{\fo} + \delta_{\fo})^2,
\end{equation}
where $*_\fo$ denotes the  Hodge star operator restricted to $\Lambda T {\mathcal F}^*$. In particular    by \cite[Corollary C]{ALK} one has that, if   $\fo$ is  a Riemannian foliation on a compact manifold with a bundle-like metric and let $V = \R$ or $\C$, then with respect to such a metric we have the leafwise Hodge decomposition
$$
\Omega^{0,v} (V) = \ker \Delta_{\fo} \oplus \overline {{\mbox{\rm im}} \, \Delta_{\fo}} = (\ker d_{\fo} \cap \ker \delta_{\fo}) \oplus  \overline{ {\mbox{\rm im}}  \, d_{\fo}} \oplus  \overline{ {\mbox{\rm im}}  \, \delta_{\fo}},
$$
where $(\Omega^{0,v} (V) , d_{\fo})$ is the leafwise de Rham complex with coefficients in $V$.

\section{Complex and K\" ahler-Riemannian  Foliations}

In this section we will deal with  special types of complex  Riemannian foliations, called K\" ahler-Riemann.
First of all we recall some basic facts of complex geometry:\\
A \emph{complex manifold} can be defined as a pair $(M,J)$, where $M$ is a $2n$-dimensional smooth manifold and $J\in{\rm End}(TM)$ satisfies
$$
J^{2}=-{\rm Id}\,,\quad [JX,JY]=J[JX,Y]+J[X,JY]+[X,Y]\,,\quad\mbox{for } X,Y\in\Gamma(TM)\,.
$$
The complex structure $J$ induces the natural splitting $TM\otimes \C=T^{1,0}M\oplus T^{0,1}M$, where
$$
\begin{aligned}
& T_{x}^{1,0}M=\{v\in T_xM\otimes \C\,\,:\,\,Jv=\operatorname{i}v\}\,;\\
& T_{x}^{1,0}M=\{v\in T_xM\otimes \C\,\,:\,\,Jv=-\operatorname{i}v\}\,.
\end{aligned}
$$
Moreover, if $\Omega^{v}$ denotes the vector bundle of complex differential forms on $M$, then
$$
\Omega^{1}=\Omega^{1,0}\oplus \Omega^{0,1}\,,
$$
where $\Omega^{1,0}$ and $\Omega^{0,1}$ are the dual vector bundles of $T^{0,1}M$ and $T^{0,1}M$, respectively.
Consequently $\Omega^{v}$ splits in
$$
\Omega^v=\bigoplus_{r+s=v}\Omega^{r,s}\,
$$
where
$$
\Omega^{r,s}=\Omega^{r,0}\otimes \Omega^{0,s}
$$
and
$$
\Omega^{r,0}=\bigwedge^r \Omega^{1,0}\,,\quad \Omega^{0,s}=\bigwedge^r\Omega^{0,1}\,.
$$

Now we consider the case of smooth foliations. 
Let $\mathcal{F}$ be a smooth $2n$-dimensional foliation on a compact  Riemannian manifold $(M, g)$. An
\emph{almost complex structure} on $\fo$ is a section of the bundle
${\rm End}(T\fo)$ of the endomorphisms of $T\fo$ satisfying $J^2=-{\rm Id}$.
If further the Nijenhuis tensor
$$
N_{J}(X,Y):=[JX,JY]-J[JX,Y]-J[X,JY]-[X,Y]\,,\quad\mbox{for } X,Y\in\Gamma(T\fo)
$$
vanishes, then $J$ is said to be a \emph{complex structure} on $\fo$ and
the pair $(\fo,J)$ is called a \emph{complex foliation}.

A Riemannian metric $g$ is  \emph{compatible}
with a complex foliation $(\mathcal{F},J)$ if
$$
g(JX,JY)=g(X,Y)\,,\quad \mbox{for any }X,Y\in\Gamma(T\fo)\,.
$$
Let us consider now a complex foliation $(\mathcal{F},J)$ on Riemannian manifold $(M,g)$ and we assume that the metric $g$ is compatible with $J$.
The tangent bundle to $\mathcal{F}$ decomposes as
$$
T\fo=T\fo^{1,0}\oplus T\fo^{0,1}\,
$$
and consequently, if we denote by $\Omega^{0,v}=\Omega^{0,v} (\C)$  the bundle of  complex  differential forms on $\mathcal{F}$, we have  the splitting
$$
\Omega^{0,v}=\bigoplus_{r+s=v}\Omega^{0,r,s}.
$$
Therefore, for the complex de Rham algebra $\Omega$ of $M$, we get the decomposition
\begin{equation} \label{omegaurs}
 \Omega^p=\bigoplus_{u+v=p}\Omega^{u,v}=\bigoplus_{u+r+s=p} \Omega^{u,r,s}\,,
\end{equation}
where the spaces $\Omega^{u,v}$ are defined by \eqref{forme}.

We denote by
$$
\begin{aligned}
&d_{i,j,k}\colon \Omega^{u,r,s}\to\Omega^{u+i,r+j,s+k}\\
&\delta_{i,j,k}\colon \Omega^{u,r,s}\to\Omega^{u+i,r+j,s+k}\,.
\end{aligned}
$$
the components of the operators $d$ and $\delta$ with respect to the above  decomposition.

A complex structure $J$ on a smooth foliation $\mathcal{F}$ induces a natural complex structure on the leafs of $\mathcal{F}$.
Furthermore the exterior  derivative  $d_\fo$ along the leaves splits as
$$
d_\fo=\partial_\fo+\ovp_\fo\,,
$$
where $\ovp_\fo$ and $\partial_\fo$ are defined by the relations
$$
(\partial_\fo \, \phi)_{|l}=\partial(\phi_{|l})\,,\quad (\ovp_\fo \, \phi)_{|l}=\ovp(\phi_{|l})\,,
$$
for any $\phi \in\Omega^{0,v}$ and any leaf $l$ of $\fo$. Note that $\partial_\fo$ and $\ovp_\fo$ can be respectively identified
with the operators
$$
\begin{aligned}
& d_{0,1,0}\colon \Omega^{0,u,v}\to\Omega^{0,u+1,v}\,,\\
& d_{0,0,1}\colon \Omega^{0,u,v}\to\Omega^{0,u,v+1}\,.
\end{aligned}
$$
Let $\vartheta_\fo$ and $\overline{\vartheta}_\fo$ be the formal adjoints of $\overline{\partial}_\fo$ and $\partial_\fo$, respectively. Then
the maps
\begin{equation} \label{box}
\begin{aligned}
&\overline{\square}_\fo=\ovp_\fo\vartheta_\fo+\vartheta_\fo\ovp_\fo\colon \Omega^{0,r,s}\to \Omega^{0,r,s}\,,\\
&\square_\fo=\partial_\fo\overline{\vartheta}_\fo+\overline{\vartheta}_\fo\partial_\fo
\colon\Omega^{0,r,s}\to \Omega^{0,r,s}\,,
\end{aligned}
\end{equation}
are elliptic leafwise differential operators.

We recall now the definition and some properties of K\"ahler-Riemann foliations (for more details  on  this topic  see
for instance  \cite{D}).

\medskip

\begin{definition} Let $(M,g)$ be a compact Riemannian manifold. A \emph{K\"ahler foliation} on $M$  is a complex foliation $(\fo,J)$   endowed with
a real $2$-form $\omega\in\Omega^{0,1,1}$ such that
$$
d_\fo \, \omega=0\,,\quad
\omega(X,Y)=g(JX, Y)\,.
$$
for any $X,Y\in \Gamma(T\fo)$.\\
If the metric $g$ is bundle-like with respect to $\fo$, then $(\fo,J,\omega)$ is called a K\"ahler-Riemann foliation.\\
\end{definition}

For  a K\"ahler-Riemann foliation $\fo$ one has  the following  K\"ahler identities
\begin{equation}
\label{Kidentities}
\Delta_\fo=2\square_\fo=2\overline{\square}_\fo.
\end{equation}

Moreover, for the cohomology groups $\overline{H}\,^{v}_{d_\fo} =\ker \Delta_\fo \cap \Omega^{0,v}$ we get the splitting
\begin{equation}
\label{Huv}
\overline{H}\,^{v}_{d_\fo}=\bigoplus_{r+s=v} \overline{H}\,^{r,s}_{\ovp_\fo},
\end{equation}
where
$$
\overline{H}^{r,s}_{\ovp_\fo}:=\ker\overline{\square}_\fo\cap\Omega^{0,r,s}.
$$
In particular if $\overline{H}\,^{v}_{d_\fo}$ is finite dimensional, then $\overline{H}\,^{r,s}_{\ovp_\fo}$ have finite dimension
for any $r,s$ such that $r + s = v$.

\section{Cosymplectic structures}

Cosymplectic structures have been introduced as an odd-dimensional analogous of K\" ahler manifolds and they can be interpreted in terms of K\" ahler-Riemann foliations.

Let start to recall the definition of a normal  almost  contact  metric structure.

\begin{definition}
Let $M$ be a $(2n+1)$-dimensional manifold. An \emph{almost contact metric} structure on $M$ consists of a quadruple $(J,\xi,\alpha,g)$, where
$J$ is an endomorphism of $TM$, $\xi$  is a tangent vector field,  $\alpha$ is a $1$-form and  $g$ is a Riemannian metric on $M$ satisfying the conditions
$$
J^2=-{\rm Id}+\alpha\otimes \xi\,,\quad \alpha(\xi)=1\,,
$$
$$
g(JX,JY)=g(X,Y)-\alpha(X)\alpha(Y), \quad X,Y\in\Gamma(TM).
$$
An almost contact metric structure $(J,\xi,\alpha,g)$ is said to be \emph{normal} if the  endomorphism $J$ satisfies
$$
N_J(X,Y)=2d\alpha(X,Y) \xi\,,\mbox{ for any }X,Y\in\Gamma(TM),
$$
where, in this case, $N_J$ is the Nijenhuis tensor defined by
$$
N_J(X,Y)=[JX,JY]-J[JX,Y]-J[X,JY]+J^2[X,Y]\,.
$$
\end{definition}

\smallskip

An almost contact metric structure is called \emph{almost cosymplectic} if $d\alpha=0$, $d\eta=0$.
Almost cosymplectic structures with K\"ahler leaves were studied in \cite{D2}.

Let $(M,J,\xi,\alpha,g)$ be an almost contact metric manifold such that $d\alpha=0$, then
$\fo=\ker\alpha$ is a smooth foliation on $M$ and the endomorphism  $J$ induces
an almost complex structure on $\fo$ which is integrable if and only if the almost contact metric structure $(J,\xi,\alpha,g)$ is normal. Therefore if $(M,J,\xi,\alpha,g)$ is a normal  almost contact metric manifold such that $d\alpha=0$, then $\ker\alpha$ is a complex  foliation on $M$ and the integral curves of $\xi$ are geodesics, and, thus, $g$ is a bundle-like metric on $M$.

Moreover, for the operators  $\square_\fo$, $\overline \square_\fo$ defined by \eqref{box} and  for the map  $\Delta_{\perp} = \delta_{-1,0} \, d_{1,0} + d_{1,0} \,  \delta_{-1,0}$ we can prove
the following
\begin{prop}\label{elliptic}
Let $(M,J,\xi,\alpha,g)$ be a  compact normal almost contact metric manifold  such that $d\alpha=0$
and let $\fo=\ker\alpha$. Then the leafwise differential operators
$$
\begin{array} {l}
\delta_{-1,0} \, d_{1,0}+\square_\fo\colon \Omega^{0,r,s}\to\Omega^{0,r,s},\\[4pt]
\delta_{-1,0} \, d_{1,0}+\overline{\square}_\fo\colon \Omega^{0,r,s}\to\Omega^{0,r,s}
\end{array}
$$
are strongly elliptic and self-adjoint.
\end{prop}
\begin{proof}
First of all we note that $\delta_{-1,0}$ does not act on $\O^{0,v}$. Hence
the operator $\delta_{-1,0} \, d_{1,0}$  acts on $\O^{0,v}$ as  $\Delta_{\bot}$ and consequently
$\delta_{-1,0} \, d_{1,0}+\square_\fo$ and  $\delta_{-1,0} \, d_{1,0}+ \overline{\square}_\fo$ are  both self-adjoint on $\O^{0,v}$.

Furthermore, we can find  a system $\{x,z^1,\dots,z^n\}$ of local coordinates around each point $p$ of $M$  such that
$$
\xi=\partial/\partial x
$$
and $\{z^1,\dots,z^n\}$ are holomorphic coordinates for the leaf passing trough $p$. Then,  we get the following  local expressions
\begin{eqnarray}
&&\delta_{-1, 0} d_{1,0} +\square_\fo=-\frac{\partial^2}{\partial
x^2}-\sum_{\overline{\mu},\beta=1}^n
g^{\overline{\mu}\beta}\frac{\partial^2}{\partial z^{\beta}\partial
z^{\overline{\mu}}}
+\mbox{ lower order terms}\,,\\
&&\delta_{-1, 0} d_{1,0} +\overline{\square}_\fo=-\frac{\partial^2}{\partial
x^2}-\sum_{\overline{\mu},\beta=1}^n
g^{\overline{\mu}\beta}\frac{\partial^2}{\partial z^{\beta}\partial
z^{\overline{\mu}}} +\mbox{ lower order terms}\,,
\end{eqnarray}
where $g^{\overline {\mu} \beta}$ denotes the inverse of the matrix representing the restriction of the metric $g$ to the leaf passing trough $p$. From the previous local formulae it follows that  the operators $\delta_{-1, 0} d_{1,0} + \square_{\fo}$ and $\delta_{-1, 0} d_{1,0} + \overline{\square}_{\fo}$ are strongly elliptic.
\end{proof}

In general the groups $\overline{H}\,^{v}_{d_\fo}$  introduced in the previous section are infinite dimensional. We will show that they are  finite dimensional for this special  type of normal almost contact  structures. Indeed, we have the following

\begin{prop}\label{normal}
Let $(M,J,\xi,\alpha,g)$ be a  $(2n + 1)$-dimensional compact normal almost contact metric manifold satisfying $d\alpha=0$ and let $\fo=\ker\alpha$.
Then the cohomology groups $\overline{H}^{v}_{d_\fo}$ have finite dimension for any $v=0,..,2n$.
\end{prop}
\begin{proof}
Let $\Theta\colon \Omega^{0,v}\to\Omega^{1,v}$ be the injective map
$$
\phi\longmapsto\alpha\wedge\phi\,.
$$
In order to prove the proposition, it is sufficient to show that $\Theta$ induces an injective map
between $\overline{H}\,^{v}_{d_\fo}$  and the de Rham cohomology group $H^{v+1}(M)$.\\
Indeed,
it is sufficient to prove that $\Theta$ takes $\Delta_\fo$-harmonic forms to $\Delta$-harmonic forms.
Let $\phi\in \overline{H}\,^{0,v}_{d_\fo}$. Then $\phi\in\Omega^{0,v}$ and satisfies the conditions
$$
\begin{cases}
d_\fo \, \phi=0,\\
\delta_\fo \, \phi=0\,,
\end{cases}
$$
which are equivalent to
$$
\begin{cases}
d\phi\wedge\alpha=0,\\
d*_\fo\phi=0\,,
\end{cases}
$$
where $*_\fo$ is the star Hodge operator on $\oplus_{v}\Omega^{0,v}$.
Since $\alpha$ is closed and $\phi$ belongs to $\Omega^{0,v}$, we have
$$
d\phi\wedge\alpha=d(\phi\wedge\alpha),\quad
*_\fo\phi=*(\phi\wedge\alpha)\,,
$$
being $*$ the Hodge star operator associated to $g$.
Hence,
$$
\Delta_\fo\phi=0\iff \Delta(\phi\wedge\alpha)=0
$$
which ends the proof.
\end{proof}

\noindent Let $(M,J,\xi,\alpha,g)$ be an almost contact metric manifold. Then we can define the fundamental form on $M$ as the  $2$-form given by
$$
\omega(\cdot,\cdot):=g(J\cdot,\cdot)\,.
$$
The $2$-form $\omega$ is $J$-invariant  and satisfies the condition
$\alpha\wedge\omega^n\neq 0$ everywhere.

\begin{definition}
A normal almost contact metric structure  $(J,\xi,\alpha,g)$ is said to be a \emph{cosymplectic structure} if the pair $(\alpha,\omega)$
satisfies
$$
d\alpha=0\,,\quad d\omega=0\,.
$$
\end{definition}

In terms of smooth foliations  note that if $(M,J,\xi,\alpha,g)$ is a cosymplectic manifold, then the pair $(\ker\alpha,\omega)$ defines a K\"ahler-Riemann foliation on $M$ (see also \cite{Kim-Pak}). Therefore, the results of the previous sections about K\"ahler-Riemann foliations can be applied.

\section{Infinitesimal deformations of cosymplectic structures}

The aim of this section is to prove Theorem \ref{main}, which is the analogous of Kodaira-Morrow theorem \cite{KM} for K\"ahler manifolds.

In order to prove  the theorem, we  review some properties about elliptic operators (for  more details see \cite{KM}
and the  references therein).\\
Let $M$ be a compact and oriented  manifold, $\mathcal{B}$ be a complex vector bundle on $M\times(-\lambda,\lambda)$
and $B_t:=\mathcal{B}_{|M\times\{t\}}$, with $t \in (-\lambda,\lambda)$. Then $\{ B_t \}$ is a  family of complex vector bundles on $M$.

Let $\psi_t\in\Gamma(B_t)$ be a family of smooth sections of $B_t$.
Then  we say that $\psi_t$  \emph{depends differentiably on $t$} if there exists a section $\psi$ of $\mathcal{B}$ such that
$$
\psi_t=\psi_{|M\times \{ t \}}\,.
$$

 Assume that every $B_t$ is
equipped with a
Hermitian metric $h_t$ on the fibres such that the family $\{h_t\}$  depends differentiably on  $t$. We recall that   a  family of linear operators $A_t\colon\Gamma(B_t)\to\Gamma(B_t)$
\emph{depends differentiably} on $t$ if the following property holds:
if $\psi_t\in\Gamma(B_t)$ depends differentiably on $t$, then also
$A_t\psi_t$ depends differentiably on $t$.

Consider  now a family of strongly  elliptic self-adjoint operators
$$
E_{t}\colon \Gamma(B_t)\to\Gamma(B_t)
$$
 depending differentiably  on $t$.

 We recall the following
\begin{theorem}  \cite[Theorem 4.3]{KM} \label{findim} Let $\mathbb{F}_t$ be  the kernel of $E_t$.
Then, the map $t\mapsto \dim \mathbb{F}_t$ is upper semicontinuous, i.e.  given a $t_0$ there exists a sufficiently small $ \epsilon$ such that $\dim \, {\mathbb{F}}_t  \leq \dim  {\mathbb{F}}_{t_0}$ for $\vert t - t_0 \vert < \epsilon$.
\end{theorem}

If we denote by
$$
F_t\colon \Gamma(B_t)\to\mathbb{F}_t
$$
the orthogonal projection with respect to $h_t$ and  by
$$
G_t\colon\Gamma(B_t)\to\Gamma(B_t)
$$
the Green operator associated to $E_t$, then we have the following
\begin{theorem} \cite[Theorem 4.5]{KM}
If  the dimension of $\mathbb{F}_t$ is independent on $t$,  then for $\vert t\vert$ sufficient small,
$F_t$ and $G_t$ depend differentiably on $t$.
\end{theorem}

 Now we  apply the previous results in the context of cosymplectic structures   in order to prove  Theorem \ref{main}.

Let $(J,\xi,\alpha,g)$ be a cosymplectic structure on a compact manifold $M$ and let $\{J_t\}_{t\in(-\lambda,\lambda)}$ be a
differentiable  family of endomorphisms of $TM$ satisfying $$J_t^2=-{\rm Id}+\alpha\otimes \xi, \, N_{J_t} =0, \, J_0=J.$$
The Riemannian  metric
$$
\widetilde{g}_t(\cdot,\cdot):=\frac12\left( g(\cdot,\cdot)+g(J_t\cdot,J_t\cdot) \right)\,,
$$
is compatible with $J_t$ and such that
$$
\widetilde{g}_t(J_t\cdot,J_t\cdot)=\widetilde{g}_t(\cdot,\cdot)-\alpha(\cdot)\alpha(\cdot)\,.
$$
Therefore,  $(J_t,\xi,\alpha, \tilde g_t)$ defines a normal almost contact metric  structure on $M$ for any $t$.

Moreover, by using the fact that $\fo_t =(\ker \, \alpha, J_t)$ defines a complex foliation on $M$ for any $t$,  and  by using  the decomposition  \eqref{omegaurs} for the complex de Rham algebra, we get the following splitting
$$
\Omega^{u,v} = \bigoplus_{r + s = v} \Omega_t^{u,r,s},
$$
where
by $\Omega_t^{u,r,s}$ we denote the vector space of complex forms in $\Omega^{u,r+s}$ which are of type $(u,r,s)$ with respect to $J_t$.

Therefore for the differential $d_{{\fo}_t}$ we have the natural decomposition
$$
d_{{\fo}_t} = \partial_{{\fo}_t}  + \ovp_{{\fo}_t}.
$$
In the sequel we will denote $\partial_{{\fo}_t}$ and $\ovp_{{\fo}_t}$ respectively by $\partial_{t}$ and $\ovp_{t}$.

Consider  the family of linear operators
$$
\widetilde{E}_t\colon\Omega_{t}^{0,r,s}\to \Omega_{t}^{0,r,s}
$$
defined by
\begin{equation}\label{Et}
\widetilde{E}_t=\partial_t\ovp_t\vartheta_t\overline{\vartheta}_t+\vartheta_t\overline{\vartheta}_t\partial_t\ovp_t
+\vartheta_t\partial_t\overline{\vartheta}_t\ovp_t+\overline{\vartheta}_t\ovp_t\vartheta_t\partial_t+\vartheta_t\ovp_t+\overline{\vartheta}_t
\partial_t,
\end{equation}
where $\overline \vartheta_t$ and $\vartheta_t$ are the $\tilde g_t$-adjoints respectively of $\partial_t$ and $\overline \partial_t$.
 Note that  $\widetilde{E}_t$ is self-adjoint with respect to the Hermitian metric $\tilde{g}_t$.

Adapting  the proof of  Proposition 4.3  of \cite{KM}  for our case we can prove

\begin{prop}\label{in0}\emph{
$\widetilde{E}_0=\square_0\square_0+\vartheta_0\ovp_0+\overline{\vartheta}_0\partial_0$ and $\widetilde{E_t}\phi=0$ if and only if
$$
\vartheta_t \overline \vartheta_t\phi=\ovp_t\phi=\partial_t\phi=0\,.
$$
}
\end{prop}
Now we introduce the new operators $E_{t}\colon \O^{0,u,v}_t \to \O^{0,u,v}_t$ defined by
$$
E_{t}=\widetilde{E}_t+\delta_{-1,0} \, d_{1,0} \, \delta_{-1,0} \, d_{1,0}+\delta_{-1,0} \, d_{1,0}\,.
$$
We have the following
\begin{prop}\label{deco}\emph{
The operator $E_t$ is  strongly elliptic  and self-adjoint.  A form
$\phi\in\Omega^{0,r,s}_t$ belongs to the space  $\mathbb{F}_t^{r,s}=\ker
E_t\cap\O_t^{0,r,s}$ if and only if
$$
\vartheta_t\overline{\vartheta}_t\phi=\partial_t\phi=\overline{\partial}_t\phi=d_{1,0}\phi=0\,.
$$
Moreover, $\phi\in{\mathbb{F}}_0^{r,s}$ if and only if
\begin{equation}\label{Et0}
\square_0\phi=d_{1,0}\phi=0\,.
\end{equation}
}
\end{prop}
\begin{proof}
Around any $p\in M$ we can find a system $\{x,z^1,\dots,z^n\}$ of local coordinates, such that $\xi=\partial/\partial x$ and
$\{z^1,\dots,z^n\}$ are holomorphic coordinates for the leaf passing
trough $p$.  By a direct computation one can show  that $E_t$ can be locally expressed
in terms of  previous  coordinates as
$$
{E}_{t}=\frac{\partial^4}{\partial x^4}+\sum_{i,j,k,l}\widetilde{g}_{t}^{\,\overline{j}i}
\widetilde{g}_{t}^{\,\overline{l}k}\frac{\partial^4}{\partial z^{i}
\partial \overline{z}^{j}\partial z^{k}\partial \overline{z}^{l}}+\mbox{lower order terms}\,.
$$
Hence ${E}_t$ is strongly elliptic. Moreover
$$
\begin{array} {lcl}
\widetilde{g}_t({E}_t \, \phi,\phi)&=&
\widetilde{g}_t(\partial_t\, \ovp_t \, \vartheta_t \, \overline{\vartheta}_t \, \phi,\phi)
+\widetilde{g}_t(\vartheta_t \, \overline{\vartheta}_t \, \partial_t \, \ovp_t \, \phi,\phi)
+ \widetilde{g}_t(\vartheta_t \, \partial_t \, \overline{\vartheta}_t \, \ovp_t \, \phi,\phi)\\[2pt]
&&+\widetilde{g}_t(\overline{\vartheta}_t \, \ovp_t \, \vartheta_t \, \partial_t \, \phi,\phi)
+\widetilde{g}_t(\vartheta_t \, \ovp_t \, \phi,\phi)
+\widetilde{g}_t(\overline{\vartheta}_t \, \partial_t \, \phi,\phi)\\[2pt]
&&+\widetilde{g}_t(\delta_{-1,0} \, d_{1,0} \, \delta_{-1,0} \, d_{1,0} \, \phi,\phi)
 + \widetilde{g}_t(\delta_{-1,0} \, d_{1,0} \, \phi,\phi)\\[4pt]
&=&
\widetilde{g}_t(\vartheta_t \, \overline{\vartheta}_t \, \phi,\vartheta_t \, \overline{\vartheta}_t \, \phi)+
\widetilde{g}_t(\partial_t \, \ovp_t \, \phi,\partial_t \, \ovp_t \, \phi)+
\widetilde{g}_t(\overline{\vartheta}_t \, \ovp_t \, \phi,\overline{\vartheta}_t \, \ovp_t \, \phi)\\[2pt]
&&+\widetilde{g}_t(\vartheta_t \, \partial_t \, \phi,\vartheta_t \, \partial_t \, \phi)
+\widetilde{g}_t(\ovp_t \, \phi,\ovp_t \, \phi)
+\widetilde{g}_t(\partial_t \, \phi,\partial_t \, \phi)\\[2pt]
&&+\widetilde{g}_t(d_{1,0} \, \delta_{-1,0} \, \phi, d_{1,0} \, \delta_{-1,0} \, \phi)
+\widetilde{g}_t(\delta_{-1,0}\, d_{1,0} \, \phi,\delta_{-1,0} \, d_{1,0} \,  \phi)\\[2pt]
&&+\widetilde{g}_t(d_{1,0}\phi,d_{1,0}\phi)\\[4pt]
&=&
\|\vartheta_t\overline{\vartheta}_t\phi\|_t^2
+\|\partial_t\ovp_t\phi\|_t^2
+\|\overline{\vartheta}_t\ovp_t\phi\|_t^2
+\|\vartheta_t\partial_t\phi\|_t^2
+\|\ovp_t\phi\|_t^2
\\[2pt]
&&
+\|\partial_t\phi\|_t^2+\|\delta_{-1,0} d_{1,0}\phi\|_t^2
+\|d_{1,0}\phi\|_t^2\,,
\end{array}
$$
which implies $E_t\phi=0$ if and only if $\vartheta_t\overline{\vartheta}_t\phi=\partial_t\phi=\overline{\partial}_t\phi=d_{1,0}\phi=0$.
Finally, the K\"ahler identities \eqref{Kidentities} imply the last part of the proposition.
\end{proof}
\noindent For the kernel $\mathbb{F}^{r,s}_t$ of $E_t$ we can show
\begin{prop} \label{kernelEt}
\emph{Let
 ${\rm
Z}^{r,s}_t:=\ker d\cap \Omega_t^{0,r,s}$. Then
$$
{\rm Z}^{r,s}_t=(\partial_t\ovp_t
\Omega_{t}^{0,r-1,s-1} \cap \ker d_{1,0})\oplus \mathbb{F}^{r,s}_t\,.
$$
}
\end{prop}
\begin{proof}
Obviously we have
$$
(\ker d_{1,0}\cap \partial_t\ovp_t \Omega_{t}^{0,r-1,s-1})\oplus
\mathbb{F}^{r,s}_t\subseteq {\rm Z}^{r,s}_t
$$
and
$$
\partial_t\ovp_t \Omega_{t}^{0,r-1,s-1}\cap \mathbb{F}^{r,s}_t= \{ 0 \}\,.
$$
On the other hand, for any $\psi\in {\rm Z}_t^{r,s}$, one has
$$
\psi=E_tG_t\psi+F_t\psi=\partial_t\ovp_t\nu+\vartheta_t\beta+\overline{\vartheta}_t\gamma+\delta_{-1,0}\mu+F_t\psi\,,
$$
where $F_t$ is the projection on the kernel of $E_t$ and $G_t$ is the Green operator associated to $E_t$.

Since $d\psi=0$, we have
$$
d(\vartheta_t\beta+\overline{\vartheta}_t\gamma+\delta_{-1,0}\mu)=0\,.
$$
Furthermore, if
$\sigma=\vartheta_t \, \beta+\overline{\vartheta}_t \, \gamma+\delta_{-1,0} \, \mu$,
then
$$
\begin{array} {lcl}
\tilde{g}_t(\sigma,\sigma)
&=&\tilde{g}_t(\vartheta_t\beta,\sigma)+\tilde{g}_t(\overline{\vartheta}_t\gamma,\sigma)+\tilde{g}_t(\delta_{-1,0}\gamma,\sigma)\\
&=&\tilde{g}_t(\beta,\ovp_t\sigma)+\tilde{g}_t(\gamma,\partial_t\sigma)+\tilde{g}_t(\gamma,d_{1,0}\sigma)=0\,.
\end{array}
$$
Hence $\sigma=0$ and the proposition follows.
\end{proof}

As a consequence we can prove

\begin{prop} \label{mainprop}
\emph{
$ \dim \mathbb{F}_{t}^{1,1}=\dim \mathbb{F}^{1,1}_0$, for $\vert t\vert$ small. Moreover, for $\vert t\vert$ sufficiently small the orthogonal projection $F_t:  \Omega^{0,2}  \to   \ker E_t \cap \Omega^{0,2}$ depends differentiably on $t$ .
}\end{prop}
\begin{proof}
By Proposition \ref{kernelEt}, we have
$$
{\rm Z}^{1,1}_t=(\partial_t\ovp_t\Omega^0\cap \ker d_{1,0})\oplus \mathbb{F}_t^{1,1}\,.
$$
Therefore
we get
$$
\mathbb{F}_t^{1,1}={\rm Z}^{1,1}_t/(\partial_t\ovp_t\Omega^0\cap\ker d_{1,0})\,.
$$
Moreover
$$
\partial_t\ovp_t\Omega^0\cap\ker d_{1,0}\subseteq\overline{d_\fo\Omega^{0,1}}\cap\ker d_{1,0}
$$
and
$$
\mathbb{F}^{1,1}_t\subseteq\frac{{\rm Z}_t^{1,1}+(\overline{d_\fo\Omega^{0,1}}\cap\ker d_{1,0})}{\overline{d_\fo\Omega^{0,1}}\cap\ker d_{1,0}}\,,
$$
where $\Omega^{0,1}  = \Omega_t^{0,1,0}  \oplus \Omega_t^{0,0,1}$ and the bar denotes the Fr\'echet closure in the space
$$
\Omega^{0,2}  = \Omega_t^{0,2,0} \oplus \Omega_t^{0,1,1} \oplus \Omega_t^{0,0,2}.
$$
Since $\fo$ is a Riemannian foliation, by Theorem \ref{ALK2}
we have the isomorphism
\begin{equation}\label{frac}
\frac{\ker d_\fo\cap\O^{0,2}\cap\ker d_{1,0}
}{\overline{d_\fo\Omega_t^{0,1}}\cap\ker d_{1,0}}\simeq\ker\Delta_\fo\cap\Omega^{0,2}\cap\ker d_{1,0}\,.
\end{equation}
Now we observe that since $d\colon\O^{0,v}\to\Omega^{1,v}\oplus\O^{0,v+1}$ splits as $d=d_{1,0}+d_{\fo}$ and
the operator $\delta_{-1,0}$ does not act on $\O^{0,u}$, then the Laplace operator
$\Delta=d\delta+\delta d$ reduces on $\O^{0,v}$ to
$$
\Delta=
(\delta_{-1,0}+\delta_{0,-1})(d_{1,0}+d_{\fo})+(d_{1,0}+d_{\fo})(\delta_{0,-1}).
$$
Moreover, a form $\phi\in\O^{0,v}$ belongs to $\ker\Delta$ if and only if it  belongs to the space $\ker d_{1,0}\cap\ker\Delta_{\fo}$. Hence the
cohomology groups
$$
H^{0,v}(M):=\ker\Delta\cap \O^{0,v}
$$
can be identified with the spaces
$$
H^{0,v}(M)=\ker(\Delta_\fo+d_{1,0})\cap\O^{0,v}=\overline{H}^{v}_{d_\fo}\cap \ker d_{1,0}\,.
$$
Now we observe that the sequence
\begin{equation} \label{esatta}
\begin{diagram}
\node{}\quad \quad \node{}\node{ \frac{{\rm Z}_t^{1,1}+ \left
(\overline{d_\fo\Omega^{0,1}}\cap \ker d_{1,0} \right)}
{\overline{d_\fo\Omega^{0,1}}\cap\ker d_{1,0}}
\hookrightarrow  \frac{\ker d \cap \Omega^{0,2}}{\overline{d_\fo\Omega^{0,1}}\cap\ker d_{1,0}}\simeq\ker \Delta\cap \O^{0,2}}\\
\node{} \node{} \node{}\arrow{s,l}{\pi_t}\\
\node{}\\
\node{} \node{} \node{\ker \square_t\cap\ker d_{1,0}\cap\O^{0,2,0}_t
\oplus \ker \overline{\square}_t \cap \ker d_{1,0}\cap\O^{0,0,2}_t}
\end{diagram}
\end{equation}
is exact, where the map $\pi_t$ is defined by
$$
\pi_t(\psi)=\pi_t(\psi_t^{0,2,0}+\psi_t^{0,1,1}+\psi_t^{0,0,2})=\psi_t^{0,2,0}+\psi_t^{0,0,2}\,,
$$
and  $\psi_t^{u,r,s}$ denotes the component of $\psi$ in $\Omega^{u,r,s}_t$.

Since $(\fo,J_0)$ is a K\"ahler-Riemann foliation, one has
the K\"ahler identities \eqref{Kidentities} and in particular
$$
\overline{H}^{2}_{\fo}=\overline{H}^{2,0}_{\ovp_0}\oplus\overline{H}^{1,1}_{\ovp_0}\oplus\overline{H}^{0,2}_{\ovp_0}\,.
$$
Moreover
$$
H^{0,2}(M)=\overline{H}^{2}_{\fo}\cap \ker d_{1,0}=\overline{H}^{2,0}_{\ovp_0}\cap \ker d_{1,0}\oplus\overline{H}^{1,1}_{\ovp_0}\cap \ker d_{1,0}\oplus
\overline{H}^{0,2}_{\ovp_0}\cap \ker d_{1,0}\,.
$$
By Proposition \ref{elliptic}, the operator $\delta_{-1,0}d_{1,0} +\overline{\square}_t$ is strongly elliptic for every $t$; hence we can apply
Theorem
\ref{findim} obtaining
$$
\dim\Big(\ker(\overline{\square}_t+d_{1,0})\cap\Omega^{0,0,2}_{t}\Big) <  \infty\,,\mbox{ for }\vert t\vert\mbox{ small}.
$$
Let
$$
\begin{array}{l}
 b_{0,v} (M) := \dim H^{0,v}(M), \\[4pt]
  h_{r,s}^{t}:=\dim\Big(\ker(\overline{\square}_t+d_{1,0})\cap\Omega^{0,r,s}_t\Big)=\dim\Big(\ker({\square}_t+d_{1,0})\cap\Omega^{0,s,r}_t\Big)\,.
  \end{array}
$$
Then, by the exactness of \eqref{esatta}, we have
\begin{equation}
\label{Ftcon}
\dim\mathbb{F}^{1,1}_t\geq b_{0,2} (M) -2h_{0,2}^{t}
\end{equation}
and furthermore, since $(J, \xi, \alpha, g)$ is cosymplectic,
\begin{equation} \label{b02}
b_{0,2} (M)=h_{2,0}^{0}+h_{1,1}^{0}+h_{0,2}^{0}=h_{1,1}^{0}+2h_{0,2}^{0}\,.
\end{equation}
By equations \eqref{Et0}  and  \eqref{b02} we obtain
$$
\dim\mathbb{F}^{1,1}_0=h_{1,1}^{0} =b_{0,2} (M)-2h_{0,2}^{0}\,.
$$
As a consequence of  Theorem \ref{findim} we have
$$
\dim\mathbb{F}_t^{1,1}\leq \dim\mathbb{F}_0^{1,1}\,,\mbox{ for $\vert t\vert$ small}\,.
$$
Since
$$
  \delta_{-1,0}  \, d_{1,0} +\overline{\square}_t\colon\Omega^{0,r,s}_t\to\Omega^{0,r,s}_t
$$
is strongly elliptic, then the map $t\mapsto h_{0,2}^{t}$ is upper semicontinuos and
$$
h_{0,2}^{t}\leq h_{0,2}^{0}\,,\mbox{ for $\vert t\vert$ small}\,.
$$
Therefore
$$
\dim\,\mathbb{F}^{1,1}_0\geq\dim\,\mathbb{F}^{1,1}_t\geq
b_{0,2} (M)-2h_{0,2}^{t}\geq b_{0,2} (M) -2h_{0,2}^{0}
=\dim\mathbb{F}^{1,1}_0\,.
$$
which implies the proposition.
\end{proof}

By using Proposition \ref{mainprop}   we are now able to  prove Theorem $\ref{main}$.

\begin{proof}[Proof of Theorem $\ref{main}$]
Let
$$
\widetilde{\omega}_{t}(\cdot,\cdot):=\widetilde{g}_t(J_t\cdot,\cdot)
$$
be the fundamental $2$-form associated to $(J_t, \tilde g_t)$. This form cannot  be  closed, but
  if we consider the new  real  $2$-form
$$
\omega_{t}:=\frac12\Big(F_t(\widetilde{\omega}_{t})+\overline{F_t(\widetilde{\omega}_{t})}\,\Big)\,,
$$
on $M$, we have that     $\omega_t$  is closed and satisfies  $\omega_0=\omega$. Since  the orthogonal projection $F_t$ depends differentiably on $t$, then
  $\alpha\wedge\omega_t^n\neq 0$ for $\vert t\vert$ sufficiently small.

Moreover,  if we introduce the  new Riemannian metric
$$
g_t(X,Y):=\omega_t(X,J_tY)+\alpha(X)\,\alpha(Y)\,, \quad X,Y\in\Gamma(TM),
$$
then  for $\vert t\vert$ small $(J_t,\xi,\alpha,g_t)$ defines a cosymplectic structure on $M$ such that $g_0 =g$ and the theorem follows.
\end{proof}
\section{Cosymplectic structures on Lie groups and new examples}

In this section we  set a one-to-one correspondence between cosymplectic and K\"ahler Lie algebras.  More precisely, we show that  if  a  K\"ahler Lie algebra is  endowed with a particular type of derivation,  then one can construct   a cosymplectic  Lie algebra.

We recall that a  \emph{cosymplectic structure} on  a real Lie algebra $\mathfrak{g}$ of dimension $2n + 1$ is a quadruple $(J,\xi,\alpha,g)$,
where $\xi$ is a non-zero vector of $\mathfrak{g}$,
$\alpha$
is  the   dual  $1$-form of $\xi$, $J$  is an endomorphism of $\mathfrak{g}$, $g$ is a inner product  on $\mathfrak g$
satisfying  the conditions $$
\begin{array} {l}
 d \alpha = 0,  \quad  J^2=-{\rm Id}+\alpha\otimes \xi,  \quad N_J =0, \\
g(J\cdot,J\cdot)=g(\cdot,\cdot)-\alpha\otimes\alpha(\cdot,\cdot), \quad
d \omega =0,
\end{array}
$$ being $\omega$  the  $2$-form defined by
$\omega(\cdot,\cdot):=g(J\cdot,\cdot)$.

If $G$ is the simply connected  Lie group with Lie algebra $\mathfrak g$,   then a  left-invariant  cosymplectic structure on  $G$  is equivalent   to  a  cosymplectic
structures on  its Lie algebra $\mathfrak g$.

For cosymplectic Lie algebras we can prove the following

\begin{theorem}  \label{constructionderivation} Cosymplectic Lie algebras   of dimension  $2n+1$ are in one-to-one correspondence with  $2n$-dimensional
K\"ahler Lie algebras   equipped with a skew-adjoint derivation $D$ commuting with its complex structure.
\end{theorem}
\begin{proof} Let $(J,\xi,\alpha,g)$ be a cosymplectic structure on a $(2n+1)$-dimensional Lie algebra $\mathfrak{g}$
and let  $ {\mathfrak h} = \ker \alpha$. Since $\alpha$ is closed, $\mathfrak{h}$
is a Lie subalgebra of $\mathfrak g$ carrying a K\"ahler structure induced by the pair $(g,J)$.
Furthermore, since $\alpha(\xi)=1$,  the vector $\xi$
does not belong to the commutator ${\mathfrak g}^1 = [ {\mathfrak g}, {\mathfrak g}]$. Therefore we have
$$
[\xi, {\mathfrak h}] \subseteq {\mathfrak h}, \quad
[{\mathfrak h}, {\mathfrak h}] \subseteq     {\mathfrak h}
$$
and, consequently,
${\mathfrak g}$ is the semidirect sum
$$
\mathfrak{g}=  \R \xi \oplus_{{\rm ad}_{\xi}} \mathfrak{h}\,.
$$
Since the fundamental $2$-form $\omega$ associated to the cosymplectic structure is closed and $\mathfrak h$ is a Lie subalgebra of $\mathfrak g$, then one has
$$
d \omega (\xi , X_1, X_2) = - \omega ([\xi , X_1], X_2) + \omega([ \xi , X_2], X_1) =0,
$$
for any $X_1, X_2 \in {\mathfrak h}$. Moreover,  by  definition of $\omega$ we have
$$
g (J {\rm ad}_{\xi} (X_1), X_2) = g (J {\rm ad}_{\xi} (X_2), X_1),
$$
for any $X_1, X_2 \in \mathfrak{h}$ and thus $J {\rm ad}_{\xi} = {\rm ad}_{\xi} J$ is a self-adjoint  endomorphism of $({\mathfrak h}, g )$.
Therefore,  ${\rm ad}_{\xi}$ is  skew-adjoint and commutes with $J$ on $\mathfrak h$..

On the other hand, if $\mathfrak{h}$ be a $2n$-dimensional Lie algebra endowed with a K\"ahler structure $(g,J)$ and a
skew-adjoint derivation $D$ commuting with $J$, then we can construct a cosymplectic Lie algebra as follows.\\
We consider the vector space $\mathfrak{g}=\R \xi \oplus \mathfrak{h}$ and we define a Lie bracket $[\,\,,\,]$
on $\mathfrak{g}$ by the relations
$$
[X,Y] =[X,Y]_{\mathfrak{h}}\,,\quad [\xi,X]=D(X)\,,
$$
for any $X,  Y \in\mathfrak{h}$, being $[\,\,,\,]_{\mathfrak{h}}$ the Lie  bracket on $\mathfrak{h}$. We extend the endomorphism $J$ on
$\mathfrak{g}$ by putting  $J \xi =0$ and we consider the
 inner product $g$ on the Lie algebra  $\mathfrak{g}$ defined as   the unique  metric extension
of the  inner product   on  $\mathfrak{h}$ for which $ \xi$ is unitary and orthogonal to $\mathfrak{h}$. Finally, as $1$-form $\alpha$   on $\mathfrak{g}$ we take the
dual to $\xi$. Then  $(\mathfrak{g},J,\xi,\alpha,g)$ is a cosymplectic Lie algebra.
\end{proof}

\begin{rem}{\em
We remark that this last Theorem is proved in a more general contest in \cite{D1}.}
\end{rem}

\begin{ex}\emph{ An example of  cosymplectic Lie algebra was given in \cite{marrero}. In this case the Lie algebra ${\mathfrak g}$ is solvable, with structure equations:
$$
[X_i, Z] = \frac{3}{2} \pi Y_i, \quad [Y_i, Z] = - \frac{3}{2} \pi X_i, \quad  i = 1, \ldots, n
$$
and the other Lie brackets are zero.  With respect
to  the  previous theorem we have  that the  K\"ahler Lie algebra ${\mathfrak h}$ is the
abelian Lie algebra ${\mbox {span}}\{X_1, Y_1, \ldots, X_n, Y_n\}$ and that $\xi=Z$.  The corresponding Lie group is unimodular and admits a compact quotient by a uniform discrete subgroup. Moreover, the  induced left-invariant Riemannian metric  on the compact quotient   is flat.
}
\end{ex}

We recall  by \cite{Milnor} that  if a  Lie group admits a compact quotient by a  uniform discrete subgroup, then the Lie group is unimodular, i.e. ${\mbox {tr}} \, ad_X =0$ for any $X$ in the Lie algebra of $G$.

By \cite{LM}, if  a unimodular Lie group $G$ admits a left-invariant symplectic structure, then it has to be solvable. Moreover,
by \cite{Hano} a K\" ahler unimodular Lie group is flat. Hence, we have the following
\begin{prop}\emph{
A cosymplectic unimodular
Lie group is  necessarily flat and solvable.}
\end{prop}

Therefore, it is interesting to see if there exist  non-flat  cosymplectic Lie algebras. By  the previous Proposition the corresponding solvable Lie group will be  not unimodular and therefore will do not admit any compact quotient. We will  construct new cosymplectic solvable Lie algebras by applying   Theorem \ref{constructionderivation} to solvable K\"ahler Lie algebras  which admit a   derivation with the previous property.

Many results are known about K\"ahler  Lie groups (see for instance \cite{cortes}).  Basic examples are flat K\" ahler Lie groups and solvable K\" ahler Lie groups
acting simply transitively on a bounded domain in $\C^n$ by biholomorphisms  (see \cite{Gindikin}). A classification of solvable K\"ahler Lie groups, up
to (a holomorphic isometry) is obtained in \cite{Pyateeskii-Shapiro}.

We recall that   by \cite{Datri,Pyateeskii-Shapiro} a normal $J$-algebra $\mathfrak a$ is in general a real  solvable Lie algebra endowed with an integrable almost complex structure $J$  and a linear form $\mu$ on $\mathfrak a$ (called admissible) such that
$$
\mu ([JX, J Y]) = \mu ([X, Y]), \quad \mu ([JZ, Z]) > 0,
$$
for all $X, Y \in {\mathfrak a}$ and $Z \neq 0 \in {\mathfrak a}$.   By the previous conditions the eigenvalues of the adjoint representation of $\mathfrak a$ are real, i.e. ${\mathfrak a}$ is completely solvable.  Moreover, the bilinear form $< X, Y> = \mu ([J X, Y])$ determines a K\"ahler  metric on ${\mathfrak a}$.

Given a solvable  K\"ahler Lie algebra $({\mathfrak h}, J, g)$, a   \lq \lq weak modification map\rq \rq \, is a linear map
${\mathcal D}: {\mathfrak h} \to {\mbox {Der}} \, {\mathfrak h}$ such that
\begin{enumerate}
\item the derivation ${\mathcal D} (X)$ is skew-adjoint with respect to $g$;
\item $[{\mathcal D} (X), J] =0$;
\item $[{\mathcal D} (X), {\mathcal D} (Y)] = {\mathcal D}  ([X, Y]) = 0$;
\item ${\mathcal D} ({\mathcal D} (X) Y - {\mathcal D} (Y) X ) =0$,
\end{enumerate}
for any $X, Y \in {\mathfrak h}$. One may define a new Lie bracket on $({\mathfrak h}, J, g)$ by
$$
(X, Y) = [X, Y] + D(X) Y - D(Y) X,
$$
which gives to $\mathfrak h$ a structure of Lie algebra and $({\mathfrak h}_D, ( \, , \, ), J, g)$ is called a
{\emph {modification}}  of $\mathfrak h$ (see \cite{Dorfmeister}).

In  \cite{Dorfmeister}  Dorfmeister proved  that any solvable   K\" ahler Lie algebra  $({\mathfrak g}, J, g)$ is a  modification of the semidirect  products ${\mathfrak a} \oplus {\mathfrak b}$, where $\mathfrak a$ is  an abelian ideal and  $\mathfrak b$ is a normal $J$-algebra.  Therefore, by using the previous result we can construct new  cosymplectic Lie algebras.

\begin{ex}  \emph{
\noindent  By \cite{Dorfmeister} there exist solvable K\"ahler Lie algebras $({\mathfrak s}, J, g)$ admitting a  skew-adjoint
derivation commuting with their complex structure $J$.
Examples of such Lie algebras
can be written in the form:
$$
{\mathfrak s} =  {\mathfrak v} + {\mathfrak s}',
$$
where  ${\mathfrak v} = \R X_0 + \R J X_0 + {\mathfrak u}$ and ${\mathfrak s}'$ are both  $J$-invariant subalgebras of ${\mathfrak s}$, such that $[J X_0, X_0] = X_0$, $[X_0, {\mathfrak s}'] = 0$ and the eingenvalues of $ad_{JX_0}$ restricted to ${\mathfrak u}$ have real part  equal to $\frac 12$. Moreover $X_0, J X_0$, ${\mathfrak u}$ and ${\mathfrak s}'$ are pairwise orthogonal.
Therefore, by \cite[Sections 5.1 and  5.2]{Dorfmeister}
${\mathfrak s}$ admits a self-adjoint derivation  $D$ commuting with $J$, defined by
$$
D( a X_0 + b JX_0 + U + X') = [JX_0, U] - \frac {1}{2} U + [JX_0, X']\,,
$$
for any $a, b \in \R$ and $U \in {\mathfrak u}, X' \in {\mathfrak s}'$. \newline
Therefore by Theorem \ref{constructionderivation} the Lie algebra $\R \xi \oplus_D {\mathfrak s}$ has a cosymplectic structure.
}
\end{ex}

\section{Cosymplectic structures on solvmanifolds}

In \cite{marrero} some examples of cosymplectic solvmanifolds are constructed as suspensions
with fibre a $2n$-dimensional compact K\"ahler manifold $N$
(a torus) of representations defined by Hermitian isometries.  Indeed, they consider
a Hermitian isometry $f$ of the torus   $N$ and define on the product $N \times \R$ the free and properly
discontinuous action of $\Z$ given by
$$
(n,(x, t))\mapsto (f^n (x),t-n)\,,
$$
for any $n \in \Z$ and $(x, t) \in N \times \R$. The  orbit space $N \times \R / \Z$ is
a compact $(2n +1)$-dimensional manifold  endowed with the cosymplectic structure induced by the natural one on  $N \times \R$.  We will show  that the solvmanifolds constructed in this way  are finite quotient of a torus and the metrics associated to the
cosymplectic structures are flat.

In this section we  will prove that any
solvmanifold admitting a cosymplectic structure, even not  left-invariant,  is a finite quotient of a torus.
By considering the universal covering $\tilde G$, it is known (see \cite{Au})
that a solvmanifold has as finite (normal) covering a solvmanifold
of the form $\tilde G / \Gamma$, with discrete isotropy subgroup $\Gamma$.
Therefore, we will consider as solvmanifold  a compact quotient of a simply-connected solvable Lie group $G$ by a uniform discrete subgroup
$\Gamma$.

In  \cite{Ha}  Hasegawa proved that a solvmanifold admits a K\"ahler structure if and only if it is a finite quotient of complex
torus which has a structure of a complex torus bundle over a complex torus. By using this  result we are able to prove  the following

\begin{theorem}\label{Giapponese}
A solvmanifold has a cosymplectic  structure if and only if it is a finite quotient of  torus which has a structure of a
torus bundle over a complex  torus.
\end{theorem}
\begin{proof} Let $M = G/ \Gamma$ be a $(2n + 1)$-dimensional solvmanifold endowed with a cosymplectic structure.
Then the product $M \times S^1$ is a solvmanifold of the form $G \times \R /(\Gamma \times \Z)$ and has
a K\"ahler structure. By \cite{Ha} $M$ is thus a finite quotient of a torus and $\Gamma \times \Z$
is an extension of a free abelian group  of rank $2l$ by the free abelian group of rank $2k$, where $2k$ is the first
Betti number of $M \times S^1$, with $k + l = n + 1$.

Therefore we have
$$
0 \to \Z^{2l -1} \to \Gamma \to \Z^{2k} \to 0,
$$
where the maximal normal free abelian subgroup of rank $2 n +1$ with finite index in $\Gamma$ is of the form $\Z^{2l -1}
\times s_1 \Z  \times s_2 \Z  \times \cdots \times s_{2k} \Z$ and the holonomy group of the Bieberbach group  $\Gamma$ is
$\Z_{s_1} \times \Z_{s_2} \times \cdots \times \Z_{s_{2k}}$, where some of the factor $\Z_{s_{i}}$ may be trivial. This follows from the fact that by \cite{Ha}  $M \times S^1 = \C^{n + 1} / ( \Gamma \times \Z)$, where   $\Gamma \times \Z$ is a Bieberbach group with holonomy $\Z_{s_1} \times \Z_{s_2} \times \cdots \times \Z_{s_{2k}}$.  Since the action of the holonomy group on $\C^{n + 1} /( \Gamma \times \Z)$ is holomorphic, $M \times S^1$  is a holomorphic fiber bundle over the complex torus $\C^k/ \Z^{2k}$ with fiber the complex torus $\C^l/ \Z^{2l}$.

In this way $M$ is a fiber bundle over the complex torus $\C^{k} / \Z^{2k}$ with fiber the torus $\R^{2l - 1}/\Z^{2l -1}$.
\end{proof}
\noindent As a direct consequence of  previous theorem
we get the following
\begin{cor}
A solvmanifold $M=G / \Gamma$ of completely solvable type   has  a cosymplectic structure if and only if it is a torus.
\end{cor}

{\bf Acknowledgements}  We would like to thank Ernesto Buzano  for useful  comments.
We also would like to thank the referee for valuable remarks and suggestions for a better presentation
of the results.


\begin{thebibliography}{12}
\bibitem{ALK}
Alvarez L\'opez, J. A., Kordyukov Y.: Long time behavior of leafwise heat flow for Riemannian foliations,
\emph{Compositio Math.} {\bf 125} (2001), 129--153.

\bibitem{Au} Auslander L.: Discrete uniform subgroups of solvable Lie groups, \emph{Trans. Amer. Math. Soc.} {\bf 99} (1961), 398--402.

\bibitem{cortes}
Baues O.,  Cortes  D.: Aspherical K\" ahler manifolds with solvable fundamental group,
\emph{Geom. Dedicata}   {\bf 122}  (2006), 215--229.

\bibitem{Blair}
Blair D. E.: \emph{Contact manifolds in Riemannian geometry}, Lecture Notes in Math.  {\bf 509}, Springer-Verlag, Berlin,   (1976).

\bibitem{BlairGo}
Blair D. E., Goldberg S. I.: Topology of almost contact manifolds,  \emph{J. Differential Geometry}  {\bf 1}  (1967), 347--354.
\bibitem{CLM}
Chinea  D., de Le\'on  M., Marrero J. C.: Topology of cosymplectic manifolds,
\emph{J. Math. Pures Appl.} (9)  {\bf 72}  (1993), no. 6, 567--591.

\bibitem{CDM}
Chinea  D., de Le\'on  M., Marrero J. C.: Spectral sequences on Sasakian and cosymplectic manifolds, \emph{Houston J. Math.}  {\bf  23} (1997),  631--649.



\bibitem{D1}
Dacko P.: On almost cosymplectic manifolds with the structure vector field $\xi$ belonging to the $k$-nullity distribution,
{\em Balkan J. Geom. Appl.}  {\bf 5}  (2000),  no. 2,  47--60.

\bibitem{D2}
Dacko P., Olszak Z.: On almost cosymplectic $(-1,\mu,0)$-spaces,  \emph{Cent. Eur. J. Math.}  {\bf 3}  (2005), no. 2,    318--330.

\bibitem{Datri} D'Atri J. E.: Holomorphic sectional curvatures of bounded homogeneous domains and related questions, {\em Trans. Amer. Math. Soc.}
{\bf 256} (1985), 405--413.

\bibitem{DeM}  De Le\'on M.,  Marrero, J.  C.: Compact cosymplectic manifolds with transversally positive definite Ricci tensor, \emph{Rend. Mat. Appl. (7)} {\bf 17} (1997),  607--624.

\bibitem{Sullivan}
Deligne P., Griffiths P., Morgan J., Sullivan D.: Real homotopy
theory of K\"ahler manifolds. \emph{Invent. Math.}  {\bf 29} (1975), no. 3,
245--274.

\bibitem{D}
Deninger C., Singhof W.: Real polarizable Hodge structures arising from foliations,  \emph{Ann. Global Anal. Geom.}  {\bf 21}  (2002),
no. 4, 377--399.





\bibitem{Dorfmeister} Dorfmeister J.: Homogeneous
K\" ahler manifolds admitting a transitive solvable group of automorphisms, {\em Annales scientifiques
de l'\'Ecole Normale
Sup\'erieure S\'er. $4$}  {\bf 18} (1985), 143--180.

\bibitem{DG} Dragomir S.,  Tomassini G.:
{\emph Differential geometry and analysis on CR manifolds},
Progress in Mathematics, {\bf 246},  Birkh\" auser Boston, Inc., Boston, MA, 2006.

\bibitem{Fujitani}  Fujitani  T.:
Complex-valued differential forms on normal contact Riemannian manifolds,
{\em Tohoku Math. J. (2)}   {\bf 18} (1966),  349--361.



\bibitem{Gindikin} Gindikin S., Piatecki\v i-\v Sapiro I.I, Vinberg  E.B.: Homogeneous K\" ahler manifolds, 1968, Geometry of homogeneous
Bounded Domains, C.I.M.E.  1967,  1-88.

\bibitem{Hano} Hano J.: On K\" ahlerian homogeneous manifolds of unimodular Lie groups, \emph{Amer. J. Math.}  {\bf 79} (1957), 885--900.

\bibitem{Ha} Hasegawa K.: A note on compact solvmanifolds with K\" ahler structures,  \emph{Osaka J. Math.} {\bf 43} (2006), 131--135.

\bibitem{Kim-Pak} Kim T. W., Pak H. K.: Canonical foliations of certain classes of almost contact metric structures,
\emph{Acta  Math. Sin.} {\bf 21} (2005), 841--846.

\bibitem{KM}
Kodaira K., Morrow J.: \emph{Complex manifolds.} Reprint of the 1971 edition with errata. AMS Chelsea Publishing,
Providence, RI, (2006). x+194 pp.

\bibitem{LM}
Lichnerowicz A.,  Medina A.:
Groupes  \'a structures symplectiques ou K\"ahleriennes invariantes, {\em C. R. Acad. Sci. Paris S\'er. I Math.}  {\bf 306}  (1988), no. 3,
133--138.




\bibitem{marrero}
Marrero J. C., Padron E.: New examples of compact cosymplectic solvmanifolds,
\emph{Arch. Math. (Brno)}  {\bf 34}  (1998),  no. 3, 337--345.

\bibitem{Milnor} Milnor J.: Curvatures of left invariant metrics on Lie groups, {\em Adv. Math.} {\bf  21} (3) (1976), 293--329.



\bibitem{Molino1}
Molino P.:  G\'eom\'etrie globale des feuilletages Riemanniens, {\em {Proc. Nederl. Acad. A1}} {\bf 85} (1982), 45--76.

\bibitem{Molino2}
Molino P.: \emph{Riemannian Foliations}, Progr. in Math. 73,
Birkh\"auser, Basel, (1988).


\bibitem{Pyateeskii-Shapiro} Pyateeskii-Shapiro I.I.: Automorphic functions and the geometry of classical domains,
Mathematics and its Applications, Vol. {\bf 8}, Gordon and Breach, New York-London-Paris, 1969.

\bibitem{Reinhart}
Reinhart B. L.: Foliated manifolds with bundle-like metrics,
\emph{Ann. of Math.} {\bf 69} (1959), 119--132.

\bibitem{sullivan}
Sullivan D.: Differential forms and the topology of manifolds. Manifolds, Tokyo 1973
(Proc. Internat. Conf. Tokyo, 1973), Univ. Tokyo Press, Tokyo, 1975,  37--49.

\end{thebibliography}
\end{document}